\documentclass[oneside, a4paper, 12 pt]{amsart}
\usepackage{amsmath,amssymb,amsthm,hyperref,adjustbox}
\usepackage{longtable,}
\usepackage{cleveref}
\usepackage{tikz-cd}
\usepackage{float}

\numberwithin{equation}{section}

\newtheorem{theorem}{Theorem}[section]

\newtheorem{lemma}[theorem]{Lemma}
\newtheorem{conjecture}[theorem]{Conjecture}
\theoremstyle{definition} 
\newtheorem*{example}{Example}

\DeclareMathOperator{\Sym}{Sym}
\usepackage{xcolor}
\definecolor{myblue}{HTML}{007ACC}
\hypersetup{
    colorlinks,
    linkcolor={myblue},
    citecolor={myblue},
    urlcolor={myblue}
    }
\title[An equivalence]{An equivalence between a conjecture of Neumann-Praeger on Kronecker classes and a conjecture on cliques of derangement graphs}
\author{Jessica Anzanello}
\address{Jessica Anzanello\\
Dipartimento di Matematica e Applicazioni\\ University of Milano-Bicocca\\
Piazza del Calendario 3, 20125 \\Milano, Italy}
\email{j.anzanello@campus.unimib.it}

\author{Pablo Spiga}
\address{Pablo Spiga\\
Dipartimento di Matematica e Applicazioni\\ University of Milano-Bicocca\\
Piazza del Calendario 3, 20125 \\Milano, Italy}
\email{pablo.spiga@unimib.it}

\begin{document}

\begin{abstract}
We prove an equivalence between a conjecture of Neumann and Praeger on Kronecker classes in algebraic number fields, and a conjecture on cliques of derangement graphs in combinatorics.
\end{abstract}


\maketitle

\section{Introduction}
\subsection{Derangement graphs}\label{der graphs}
Let $G$ be a group acting on a finite set $\Omega$. An element $g \in G$ is called a \emph{derangement} if it has no fixed points on $\Omega$.
If $G$ is transitive, and $|\Omega|\ge2$, a classical theorem of Jordan \cite{Jordan} guarantees the existence of a derangement.
This elementary fact has motivated a wide range of problems and applications across various areas of mathematics, ranging from number theory to topology; see, for example, Serre's account \cite{Serre}; and in recent years the study of derangements has become a central topic in permutation group theory; see  the introductory chapter of \cite{BurGiu} for a survey.

In this paper we focus on a combinatorial object naturally associated with $G$, namely, the \emph{derangement graph} $\Gamma_G$.
This is the Cayley graph of $G$ with connection set equal to the set $\mathcal{D}(G)$ of all derangements of $G$: the vertex set is $G$, and two vertices $x,y \in G$ are adjacent if and only if $xy^{-1}$ is a derangement. Note that $\Gamma_G$ is loopless, since $\mathcal{D}(G)$ does not contain the identity element of $G$, and it is a simple graph because  $\mathcal{D}(G)$  is inverse-closed, that is, $\mathcal{D}(G)=\{d^{-1}\mid d \in \mathcal{D}(G)\}$. Moreover, in many natural situations one has $G=\langle \mathcal{D}(G) \rangle$, and therefore $\Gamma_{G}$ is connected. For instance, if $G$ is simple, then $\langle \mathcal{D}(G) \rangle$ is a non-trivial normal subgroup of $G$, and therefore equals $G$. Moreover,
Burness and Fusari \cite{BurFus}  recently proved that a finite simple transitive group can be generated by two conjugate derangements, and precise upper bounds for the index $[G:\langle \mathcal{D}(G) \rangle]$ in transitive permutation groups
have been investigated in \cite{BaiCa, Garzoni}.
Here, we are interested in the information encoded by cliques with respect to the action of $G$. Before discussing this, we remark some facts about the complementary notion of cocliques. Cliques and cocliques in $\Gamma_G$ are related by the so-called clique-coclique bound (see~\cite[Theorem~2.1.1]{GodsilMeagher}):
denoting by $\omega(\Gamma_G)$ and $\alpha(\Gamma_G)$ the clique number and the coclique (or independence) number of $\Gamma_G$, that is, the maximum cardinality of a clique and of a coclique in 
$\Gamma_G$, we have $\alpha(\Gamma_G)\omega(\Gamma_G)\le |G|$. Cocliques, also known as \textit{intersecting sets}, are closely related to the classical notion of intersecting families in extremal combinatorics, whose study culminated in the Erd\H{o}s--Ko--Rado theorem, and whose extensions to permutation groups have been studied for example in \cite{6,10, li2020ekr,KRS}. 

As an elementary fact, if $G$ is transitive and $|\Omega| \geq 2$, then Jordan’s theorem guarantees the existence of a derangement, so $\Gamma_G$ contains a clique of size $2$.
Moreover, it was shown in \cite[Theorem~1.5]{KRS} that if $|\Omega| \geq 3$, then $\Gamma_G$ contains a triangle.
These results suggest that the absence of large cliques in $\Gamma_G$ strongly restricts the possible degree of the action. Motivated by this phenomenon, Meagher, Razafimahatratra and the second author \cite[Question~6.1]{KRS} formulated the following conjecture.

\begin{conjecture}\label{conjecture1}
There exists a function $F$ such that, if $G$ is a transitive permutation group of degree $n$ whose derangement graph contains no clique of size $c$, then $n~\leq~F(c)$.
\end{conjecture}

For small values of $c$, the conjecture is known to hold.
As discussed above, when $c = 2$, Jordan’s theorem implies $n \leq 1$, and when $c = 3$, \cite[Theorem~1.5]{KRS} gives $n \leq 2$.
Moreover, for $c = 4$, the main theorem of \cite{CLS} shows that $n \leq 30$, and that this bound is best possible.
More generally, Conjecture~\ref{conjecture1} was proved in \cite{FPS} for primitive groups and, in fact, for the wider class of innately transitive groups, that is groups $G\le \Sym(\Omega)$ admitting a minimal normal subgroup acting transitively on $\Omega$. 
The main result of this paper is a surprising equivalence between Conjecture~\ref{conjecture1} and a long-standing number-theoretic conjecture on Kronecker classes, which we discuss in the next section. In addition, in order to prove this equivalence, we establish a weaker version of Conjecture~\ref{conjecture1}, in which the bound on $n$ is allowed to depend also on an additional parameter $\ell$ (see Theorem~\ref{thrm:1}). 
We conclude this section with a remark. Traditionally, the derangement graph is defined only for faithful actions, so that 
$G$ is a permutation group on 
$\Omega$. In our setting, however, it is convenient to allow arbitrary group actions. This added flexibility is essential for our inductive arguments, where one frequently encounters a transitive group acting on a system of imprimitivity, and needs to lift cliques from the derangement graph associated with the induced action to the derangement graph corresponding to the action on $\Omega$.

\subsection{Normal coverings and Kronecker classes}\label{coveringsKronecker}Let $k$ be an algebraic number field, and let $K/k$ be a finite extension. Since the late nineteenth century, beginning with the work of Kronecker \cite{Kronecker}, a central theme in algebraic number theory has been the problem of characterising finite extensions $K/k$ of algebraic number fields in terms of the decomposition behaviour of the prime ideals of  $k$ in $K$. We briefly recall some basic facts (see \cite{Neukirch}, or any standard textbook in algebraic number theory, for background).
Let $\mathcal{O}_k$ and $\mathcal{O}_K$ denote the rings of algebraic integers of $k$ and $K$, respectively.
If $\mathfrak{p}$ is a prime ideal in $\mathcal{O}_k$, then its extension to $\mathcal{O}_K$ admits a unique factorisation
$$\mathfrak{p}\mathcal{O}_K=\prod_{i=1}^{r}\mathfrak{P}_i^{e_i},$$
where the $e_i$ are positive integers, and the $\mathfrak{P}_i$ are the distinct prime ideals of $K$ lying above $\mathfrak{p}$, that is, satisfying $\mathfrak{P}_i\cap \mathcal{O}_k=\mathfrak{p}.$
For each prime $\mathfrak{P}_i$ above $\mathfrak{p}$ one may consider the induced extension of residue fields $
\mathcal{O}_k/\mathfrak{p} \subseteq \mathcal{O}_K/\mathfrak{P}_i$,
whose degree
\[
f(\mathfrak{P}_i/\mathfrak{p}) = 
\bigl[\mathcal{O}_K/\mathfrak{P}_i : \mathcal{O}_k/\mathfrak{p}\bigr]
\]
is called the \emph{residue degree} of $\mathfrak{P}_i$ over $\mathfrak{p}$.

The \emph{\textbf{Kronecker set}} of $K$ over $k$ is the set of all prime ideals 
$\mathfrak{p} \subset \mathcal{O}_k$ for which there exists at least one prime ideal 
$\mathfrak{P}$ of $\mathcal{O}_K$ lying above $\mathfrak{p}$ such that 
$f(\mathfrak{P}/\mathfrak{p}) = 1$.
Kronecker \cite{Kronecker} proved that a Galois extension of prime degree is characterized by its Kronecker set, and in 1903 Bauer \cite{Bauer} showed that the same holds for all finite Galois extensions. However, in 1926 Gassmann \cite{Gassmann} produced an example of two nonconjugate extensions of $\mathbb{Q}$, each of degree 180, having the same Kronecker set. This example effectively halted progress in the area for several decades, until renewed interest emerged in the 1960s, when Jehne \cite{Jehne} introduced the concept of \textit{\textbf{Kronecker equivalence}}:
two finite extensions of $k$ are said to be Kronecker equivalent if their Kronecker sets differ only on finitely many primes. This defines an equivalence relation and partitions extensions into \textit{\textbf{Kronecker classes}}. Extensions belonging to the same class share strong arithmetic similarities.
The group-theoretic interpretation of Kronecker equivalence is described in~\cite{Jehne,Klingen1,Klingen2,Praeger}. Let $K$ and $K'$ be finite extensions of $k$, and let $M$ be a common Galois closure over $k$ containing both fields. Set $G=\mathrm{Gal}(M/k)$, $U=\mathrm{Gal}(M/K)$, and $U'=\mathrm{Gal}(M/K')$. Then $K$ and $K'$ are Kronecker equivalent if and only if
\begin{equation}\label{perlis}
\bigcup_{g\in G}U^g=\bigcup_{g\in G}(U')^g.
\end{equation}

This criterion makes the link with permutation group theory explicit. Indeed, viewing $G$ acting by right multiplication on the cosets of $U$ and $U'$, condition~\eqref{perlis} is equivalent to saying that the two corresponding permutation representations of $G$ have the same set of derangements.

Several problems on Kronecker equivalence have been tackled using finite group theory. In particular, examples of infinite Kronecker classes constructed by Jehne and further results discussed  by Praeger in ~\cite{Praeger2} suggest
that if $K/k$ is an extension of algebraic number fields of degree $n$ and $L$ is Kronecker equivalent to $K$, then $~[L:k]$ should be bounded by a function of $n$. In group-theoretic terms, this problem was formulated by Neumann and Praeger in 1988 in the following conjecture.

\begin{conjecture}[{{Neumann-Praeger, see~\cite{Praeger2}}}]\label{conjecturePraeger0}
There exists a function $f$ such that if $G$ is a finite group with subgroups $U$ and $U'$ satisfying $|G:U'|=n$ and
$$\bigcup_{g\in G}U^g=\bigcup_{g\in G}(U')^g,$$
then $|G:U|\le f(n)$.
\end{conjecture}
See also the Kourovka Notebook~\cite[11.71]{notebook}.

\subsection{Our main results}
The principal result of this paper reveals an unexpected connection between Conjectures~\ref{conjecture1} and~\ref{conjecturePraeger0}.

\begin{theorem}\label{thrm:main}
Conjecture~$\ref{conjecture1}$ holds if and only if Conjecture~$\ref{conjecturePraeger0}$ holds.
\end{theorem}

The proof of this equivalence is not immediate and relies on an intermediate result which is of independent interest.

\begin{theorem}\label{thrm:1}
There exists a function $h:\mathbb{N}\times\mathbb{N}\to\mathbb{N}$ such that if $G$ is a transitive permutation group of degree $n$ whose derangement graph has no cliques of size $c$, then $n\le h(c,\ell)$, where $\ell$ denotes the maximum length of a normal imprimitivity series of $G$.
\end{theorem}

Here, a system of imprimitivity $\Sigma$ for $G$ is called \textit{\textbf{normal}} if there exists a normal subgroup $N\unlhd G$ whose orbits are precisely the blocks of $\Sigma$. Given two such systems $\Sigma$ and $\Sigma'$, we write $\Sigma'\le \Sigma$ to mean that $\Sigma'$ refines $\Sigma$. A \textit{\textbf{normal imprimitivity series}} of $G$ is a series
$$\Sigma_\ell<\Sigma_{\ell-1}<\cdots<\Sigma_1,$$
where $\Sigma_i$ is a normal system of imprimitivity, for each $i$.

Theorem~\ref{thrm:1} can be viewed as a combinatorial analogue of a result of Praeger~\cite[Theorem~4.6]{praeger}. Indeed, Praeger proved that Conjecture~\ref{conjecturePraeger0} holds if the function $f$ is allowed to depend not only on $n$, but also on the length of a chief series of $G$. In Theorem~\ref{thrm:1}, the role of a chief series is played instead by the length of a normal imprimitivity series. This resemblance is, however, mostly superficial and aesthetic, since the methods of proof are substantially different. (For the interested reader wishing to consult Praeger’s proof, we note that it contained a minor error, which was later corrected in~\cite{FHS}.)

The proof of Theorem~\ref{thrm:1} is also far from straightforward. One of its main ingredients is a result on elements of prime order in finite simple groups of Lie type~\cite{AS}, which in turn relies on deep number-theoretic properties of Lehmer numbers.

We do not attempt to optimize the function $h(c,\ell)$ appearing in Theorem~\ref{thrm:1}, since our proof ultimately relies on certain number theoretic results used in the main results in~\cite{AS}, for which no effective bounds are currently available, to the best of our knowledge.

We conclude this introductory section with a brief non-mathematical remark that we cannot resist making. We find it striking that a seemingly naive question such as~\cite[Question~6.1]{KRS} on derangement graphs conceals deep connections with algebraic number theory via Kronecker classes, with number theory through Lehmer numbers, and with the classification of finite simple groups.

\section{Auxiliary results}\label{sec:auxiliary}
We first recall a theorem of Saxl.
\begin{theorem}[{\cite[Proposition~2]{Saxl}}]
\label{thrm:saxl}Let $T$ be a non-abelian simple group and let $M$ be a subgroup of $T$. If
$$T=\bigcup_{\varphi\in\mathrm{Aut}(T)}M^\varphi,$$
then $T=M$.
\end{theorem}

Next we recall one of the main results from~\cite{AS}. Let $T$ be a group.
Given $x \in T$, we define $\mathrm{mpr}^\ast(x)$ to be the number of generators in $\langle x \rangle$ belonging to distinct $\mathrm{Aut}(T)$-classes, that is
$$\mathrm{mpr}^*(x)=|\{ y^{\mathrm{Aut}(T)}\mid \langle x\rangle=\langle y\rangle\}|.$$
Also, we let ${\bf o}(x)$ denote the order of $x$. 
Given a prime number $p$, we let
\[
\mathrm{mpr}_p^\ast(T) = \max_{\substack{x \in T \\ {\bf o}(x) = p}} \mathrm{mpr}^\ast(x)\hbox{ and }
\mathrm{mpr}^\ast(T) = \max_{p\ \textrm{prime}} \mathrm{mpr}_p^\ast(T).
\]
\begin{theorem}[{\cite[Theorem~1.3]{AS}}]\label{thrm:4}
There exists an increasing function $f_{Lie}:\mathbb{N}\to \mathbb{N}$ such that, for a finite simple group of Lie type $T$, the order of $T$ is at most $f_{Lie}(\mathrm{mpr}^\ast(T))$.
\end{theorem}

We also require~\cite[Theorem~1.1]{FPS}, but only for the special case where the group is quasiprimitive. The full statement in~\cite{FPS} is considerably broader, as it concerns innately transitive groups. Recall that a permutation group $G$ is said to be \textit{\textbf{innately transitive}} if $G$ possesses a minimal normal subgroup that acts transitively, whereas $G$ is \textit{\textbf{quasiprimitive}} if every non-identity normal subgroup of $G$ acts transitively.

\begin{theorem}[\cite{FPS}, Theorem~1.1]\label{thrm:FPS}
There exists a function $f_1:\mathbb{N}\to\mathbb{N}$ such that, if $G$ is innately transitive of degree $n$ and the derangement graph of $G$ has no clique of size $c$, then $n\le f_1(c)$.
\end{theorem}

We also need some number theoretic results. For $b \in \mathbb{Z}$ with $b \ne 0$, let $P[b]$ denote the largest prime divisor of $b$, with the convention that $P[b]=1$ when $b \in \{\pm 1\}$.

\begin{lemma}[{\cite[Equation~(1)]{Siegel}}]\label{siegeltheorem}
Let $f(x)\in\mathbb{Z}[x]$ be a polynomial with at least two distinct roots. Then there exist two positive constants $c_f$ and $c_f'$ depending on $f$ only such that $P[f(q)]\ge c_f\log \log q$ for every $q\ge c_f'$.
\end{lemma}

In the proof of Lemma~\ref{lem:aux1}, we need the following result of Liebeck, Praeger and Saxl.
\begin{theorem}[{\cite[Theorem~4]{LPS}}]\label{thrm:4LPS}
Let $T$ be a finite non-abelian simple group which is not an alternating group and let $M$ be a proper subgroup of $T$. Suppose that $|M|$ is divisible by each of the primes or prime powers indicated in the second or third column in Tables~$10.1-10.6$. Then the possibilities for $M$ are given in Tables~$10.1-10.6$.
\end{theorem}

Before stating the final result of this section, we introduce some notation.

Let $p$ be a prime, and let $e,d$ be positive integers with $d\ge 2$. Set $q=p^{e}$. We assume that $T$ is one of the following groups:
\begin{itemize}
\item $T=\mathrm{PSL}_d(q)$, with $(q,d)\notin\{(2,2),(3,2)\}$;
\item $T=\mathrm{PSp}_d(q)'$, with $d$ even and $d\ge 4$;
\item $T=\mathrm{PSU}_d(q)$, with $d\ge 3$ and $(d,q)\ne (3,2)$;
\item $T=\Omega_d(q)$, with $d\ge 7$ and $dp$ odd;
\item $T=\mathrm{P}\Omega_d^\pm(q)$, with $d\ge 8$ and $d$ even;
\item $T\in\{E_6(q),E_7(q),E_8(q),{}^2E_6(q),F_4(q),{}^2F_4(q)',{}^3D_4(q)\}$, or 
$T={}^2B_2(q)$ with $q=2^e$ and $e\ge 3$ odd, or 
$T={}^2G_2(q)$ with $q=3^e$ and $e\ge 3$ odd.
\end{itemize}

Let $\Phi_n(x)$ be the $n$-th cyclotomic polynomial.
We follow the notation of~\cite{LPS} and, for each $\ell\in\mathbb{N}$, write $q_{\ell}=P[\Phi_{\ell e}(p)]$.

\begin{lemma}\label{lem:cats}
There exists a function $\lambda:\mathbb{N}\to\mathbb{R}$ with the following property.  
Let $T$ be a finite simple group of Lie type, let  $q$ be defined as above, and let $r$ be the Lie rank.  
For every $\ell\ge 1$, if $q_{\ell}$ divides $|T|$, then either
$|T|\le \lambda(q_{\ell},r)$
or $\ell e\le 2$.
\end{lemma}

\begin{proof}Since $|T|$ is determined by $r$ and $q$ (see~\cite[page~xvi, Table~6]{atlas}), it suffices to show that $q$ is bounded above by a function of $q_\ell$.

Assume that $\ell e \ge 3$. Then the cyclotomic polynomial $\Phi_{\ell e}$ has degree at least $2$, and therefore possesses at least two distinct roots.  

Given $x\in\mathbb{N}$, let 
\begin{align*}
c_x'&=\max\{c_{\Phi_k}\mid 3\le k\le x-1\},\\
c_x&=\min\{c_{\Phi_k}\mid 3\le k\le x-1\},
\end{align*}
where $c_{\Phi_{k}}$ and $c_{\Phi_{k}}'$ are the constants arising from Lemma~\ref{siegeltheorem} applied with the polynomial $f=\Phi_{k}$.
Next define the function $\kappa:\mathbb{N}\to\mathbb{R}$ by
\[
\kappa(x)=\max\left\{c_{x}',\ \exp\!\big(\exp(x/c_{x})\big)\right\},
\]
From Fermat's little theorem, $\ell e\le q_\ell-1$ and hence, by Lemma~\ref{siegeltheorem}, we obtain $q\le \kappa (q_\ell)$. 
\end{proof}

\begin{lemma}\label{lem:aux2}
Let $s,a$ and $\sigma$ be positive integers, let $X$ be a set of cardinality $s$, and let $\pi_1,\ldots,\pi_\sigma$ be a family of partitions of $X$ such that each part in $\pi_i$ has cardinality at least $a$, for every $i\in \{1,\ldots,\sigma\}$. Then there exists a subset $Y\subseteq X$ with 
$|Y|\ge s(1-1/a)^\sigma$ such that $Y$ contains no part of $\pi_i$ for any $i\in \{1,\ldots,\sigma\}$.
\end{lemma}

\begin{proof}
We argue by induction on $\sigma$.

Assume first that $\sigma=1$. Write $\pi_1=\{X_1,\ldots,X_b\}$ and choose an element $x_i\in X_i$ for each $i$. 
Set $Y=X\setminus\{x_1,\ldots,x_b\}$. Then $$|Y|=s-b\ge s-\frac{s}{a}=s(1-1/a)$$ and, by construction, $Y$ contains no part of $\pi_1$. This proves the result for $\sigma=1$.

Now assume $\sigma>1$. By the inductive hypothesis, there exists a subset $Y'\subseteq X$ with
$|Y'|\ge s(1-1/a)^{\sigma-1}$ such that $Y'$ contains no part of $\pi_i$ for every $i\in\{1,\ldots,\sigma-1\}$.  

The set $Y'$ contains at most $\lfloor |Y'|/a \rfloor$ parts of $\pi_\sigma$. By removing at most this many elements from $Y'$, we obtain a subset $Y\subseteq X$ that contains no part of $\pi_\sigma$. Hence
\begin{align*}
|Y|
&\ge |Y'|-\lfloor |Y'|/a \rfloor
\ge |Y'|-\frac{|Y'|}{a}
= |Y'|\left(1-\frac{1}{a}\right) \\
&\ge s\left(1-\frac{1}{a}\right)^{\sigma-1}\left(1-\frac{1}{a}\right)
= s\left(1-\frac{1}{a}\right)^\sigma.\qedhere
\end{align*}
\end{proof}

\section{Proof of the equivalence: Theorem~\ref{thrm:main}}\label{proof1}
In this section, we assume that Theorem~\ref{thrm:1} holds true and we prove Theorem~\ref{thrm:main}.

 Assume that Conjecture~\ref{conjecture1} holds.  
Let $G,U,U'$ be as in Conjecture~\ref{conjecturePraeger0}, and let $\Omega$ be the set of right cosets of $U$ in $G$. Then
$$\bigcup_{g\in G} U^g$$
is precisely the set of elements of $G$ that fix at least one point of $\Omega$.

If this union coincides with $\bigcup_{g\in G} U'^g$ and $|G:U'| = n$, then any clique in the derangement graph of $G$ acting on $\Omega$ has size at most $n$. Indeed, suppose that $C$ is a clique with $|C| > n$. By the pigeonhole principle, $C$ contains two elements $x,y$ lying in the same right coset of $U'$. Then $xy^{-1} \in U'$, and hence $xy^{-1}$ is conjugate to an element of $U$. Consequently, $xy^{-1}$ fixes some point of $\Omega$, contradicting the fact that $C$ is a clique.

Therefore, the derangement graph of $G$ on $\Omega$ has no clique of size $n+1$. By Conjecture~\ref{conjecture1}, we obtain $|\Omega|=|G:U| \leq F(n+1)$, which proves Conjecture~\ref{conjecturePraeger0}.

We now prove the converse.   Assume that Conjecture~\ref{conjecturePraeger0} holds true. Let $f$ be the function appearing in Conjecture~\ref{conjecturePraeger0} and let $h$ be the function appearing in Theorem~\ref{thrm:1}. Now let $f':\mathbb{N}\to\mathbb{N}$ and $h':\mathbb{N}\times\mathbb{N}\to\mathbb{N}$ be defined by
\begin{align*}
f'(n)&=\max\{\lceil f(x)\rceil\mid x\le n\}+n,\\
h'(c,n)&=\max\{\lceil h(c,x)\rceil\mid x\le n\}+n,
\end{align*}
for every $n,c\in\mathbb{N}$. Clearly, $f'$ and $h'(c,-)$ are increasing and since $f(n)\le f'(n)$ and $h(c,n)\le h'(c,n)$, replacing $f$ and $h$ if necessary, we may suppose that $f$ and $h(c,-)$ themselves are increasing.

 We define recursively a family of functions $f_i:\mathbb{N}\to\mathbb{R}$ indexed by the positive integers. We let $f_0(x)=0$ and $f_1(x)=1$, for every $x$. Now, assume that $f_{j-1}(x)$ has been defined and let 
\begin{equation}\label{functions}
f_j(c)=f_{j-1}(c)+\log_2(f(h(c,f_{j-1}(c)))!).
\end{equation} Now, let
\begin{equation*}
F(c)=h(c,f_{\lceil\log_2(c)+1\rceil}(c)).
\end{equation*}
We show that Conjecture~\ref{conjecture1} holds true with this choice of $F$.

Let $G$ be a transitive group of degree $n$ whose derangement graph has no clique of size $c$ and let $\Omega$ be the domain of $G$. Let $$\Omega=\Sigma_\ell<\Sigma_{\ell-1}<\cdots<\Sigma_1<\Sigma_0=\{\Omega\}$$
be a normal  imprimitivity series for $G$ of maximal length. Associated to this chain we have an analogous chain of normal subgroups of $G$:
$$1=G_{(\Sigma_\ell)}< G_{(\Sigma_{\ell-1})}< \cdots < G_{(\Sigma_1)}< G_{(\Sigma_0)}=G,$$
where $G_{(\Sigma_i)}$ is the kernel of the action of $G$ on $\Sigma_i$. As $\Sigma_i$ is a normal system of imprimitivity, the orbits of $G_{(\Sigma_i)}$ on $\Omega$ are precisely the elements of $\Sigma_i$.

Observe that, as the derangement graph of $G$ has no clique of size $c$ in its action on $\Omega$, neither does the derangement graph of $G$ in its action on $\Sigma_i$. Therefore,
for every $i\in \{0,\ldots,\ell\}$, from Theorem~\ref{thrm:1}, we have
\begin{equation*}
[G:G_{(\Sigma_i)}]|\le |\Sigma_i|!\le h(c,i)!.
\end{equation*}

Let $0=i_0<i_1<\cdots <i_\kappa=\ell$ be a sequence of maximal length with the property that, for each $j\in \{0,\ldots,\kappa-1\}$,
\begin{enumerate}
\item\label{cond1} $G_{(\Sigma_{i_j})}$ has a derangement $g_{i_j}$ in its action on $\Sigma_{i_{j+1}}$, and
\item\label{cond2} $G_{(\Sigma_{i_j})}$ has no derangement in its action on $\Sigma_x$ when $x<i_{j+1}$.
\end{enumerate}
For instance, since $G_{(\Sigma_0)}=G$ is transitive on $\Sigma_{x}$, for every $x>0$, from Jordan's theorem, we have $i_1=1$. Next, if $G_{(\Sigma_1)}$ has no derangement on $\Omega=\Sigma_\ell$, then $\kappa=1$ and $(i_0,i_1)=(0,1)$. Whereas, if $G_{(\Sigma_1)}$ has a derangement in its action on $\Omega=\Sigma_{\ell}$, then we choose $i_2\le \ell$ minimal with the property that $G_{(\Sigma_1)}$ has a derangement in its action on $\Sigma_{i_2}$.

Let $$C=\{g_{i_0}^{\varepsilon_0}g_{i_1}^{\varepsilon_1}\cdots g_{i_{\kappa-1}}^{\varepsilon_{\kappa-1}}\mid \varepsilon_{i_j}\in \{0,1\}, \forall j\in \{0,\ldots,\kappa-1\}\}.$$
We show that $C$ is a clique for the derangement graph of $G$ in its action on $\Omega$. Let $h=g_{i_0}^{\varepsilon_0}g_{i_1}^{\varepsilon_1}\cdots g_{i_{\kappa-1}}^{\varepsilon_{\kappa-1}}$ and $h'=g_{i_0}^{\varepsilon_0'}g_{i_1}^{\varepsilon_1'}\cdots g_{i_{\kappa-1}}^{\varepsilon_{\kappa-1}'}$ be two distinct elements of $C$. Since $h$ and $h'$ are distinct, there exists $x\in \{0,\ldots,\kappa-1\}$ with $\varepsilon_x\ne \varepsilon_x'$ and $\varepsilon_y=\varepsilon_y'$, for all $y<x$. 
 As $\varepsilon_x-\varepsilon_x'\in \{1,-1\}$, from~\eqref{cond1}, $hh'^{-1}$ is conjugate to
 $g_{i_x}^{\pm 1}$ modulo $G_{(\Sigma_{i_{x+1}})}$. By~\eqref{cond2}, $g_{i_x}^{\pm 1}$ a derangement in its action on $\Sigma_{i_{x+1}}$ and hence so is $hh'^{-1}$. Therefore, $hh'^{-1}$ is also a derangement in its action on $\Omega$.

Since the derangement graph of $G$ in its action on $\Omega$ has no cliques of size $c$ and $|C|=2^{\kappa-1}$, we get 
\begin{equation}\label{eq_1}
\kappa\le \log_2 (c)+1.
\end{equation}

We show that $i_j\le f_j(c)$, for every $j\in \{0,\ldots,\kappa\}$, where the function $f_j$ is defined in~\eqref{functions}. We argue by induction. Let $\bar\alpha\in \Omega$ and, for every $x\in \{0,\ldots,\ell\}$, let $\Delta_x\in\Sigma_x$ with $\bar\alpha\in \Delta_x$. Assume first that $j\in\{0,1\}$. Thus $i_j=j$ and we have defined $f_j(c)=j$, for every $c$. Assume now $j>1$ and $i_{j-1}\le f_{j-1}(c)$. By~\eqref{cond2}, $G_{(\Sigma_{i_{j-1}})}$ has no derangement in its action on $\Sigma_{i_j-1}$. This implies that
$$G_{(\Sigma_{i_{j-1}})}=\bigcup_{g\in G}(G_{(\Sigma_{i_{j-1}})}\cap G_{\{\Delta_{i_j-1}\}})^g,$$
because $G_{(\Sigma_{i_{j-1}})}\cap G_{\{\Delta_{i_j-1}\}}$ is 
the stabilizer in $G_{(\Sigma_{i_{j-1}})}$ of the block $\Delta_{i_{j}-1}\in \Sigma_{i_j-1}$. 
From Conjecture~\ref{conjecturePraeger0} applied with $U=G_{(\Sigma_{i_{j-1}})}$ and $U'=G_{(\Sigma_{i_{j-1}})}\cap G_{\{\Delta_{i_j-1}\}}$, we deduce
\begin{equation}\label{aux}
[G:G_{(\Sigma_{i_{j-1}})}\cap G_{\{\Delta_{i_j-1}\}}]\le f([G:G_{(\Sigma_{i_{j-1}})}]).
\end{equation}
As $G/G_{(\Sigma_{i_{j-1}})}$ is a permutation group on $\Sigma_{i_{j-1}}$ whose derangement graph has no clique of size $c$ and whose maximal length of a normal imprimitivity series is $i_{j-1}$, from Theorem~\ref{thrm:1}, we have
\begin{equation}\label{thursday}
f([G:G_{(\Sigma_{i_{j-1}})}])\le f(h(c,i_{j-1})!)\le f(h(c,f_{j-1}(c))!),
\end{equation}
where the last inequality follows from the fact that $i_{j-1}\le f_{j-1}(c)$. Observe  $\Delta_{i_j-1}\subseteq \Delta_{i_{j-1}}$. Figure~\ref{fig:blocks} illustrates this construction.
 The group $G_{(\Sigma_{i_{j-1}})}$ is transitive on $\Delta_{i_{j-1}}$, because by hypothesis our systems of imprimitivity are normal. Hence $G_{(\Sigma_{i_{j-1}})}$ is transitive on the blocks of $\Sigma_{i_{j}-1}$ contained in $\Delta_{i_{j-1}} $. The number of these blocks is $|\Delta_{i_{j-1}}|/|\Delta_{i_j-1}|$ and hence, from the orbit-stabilizer theorem, we have
$$[G_{(\Sigma_{i_{j-1}})}:G_{(\Sigma_{i_{j-1}})}\cap G_{\{\Delta_{i_j-1}\}}]=\frac{|\Delta_{i_{j-1}}|}{|\Delta_{i_j-1}|}.$$

\begin{figure}
 \centering
    \includegraphics[width=9cm]{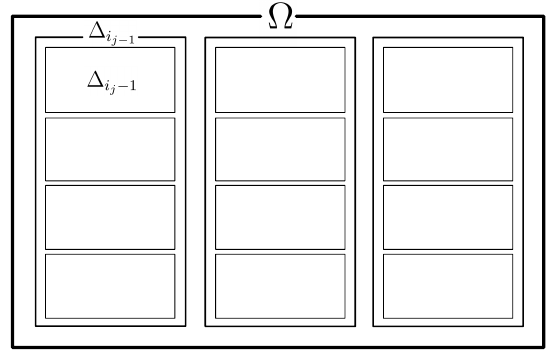}
    \caption{\textbf{Example}: Normal systems of imprimitivity $\Sigma_{i_{j-1}}$ (thick lines) and $\Sigma_{i_j-1}$ (thin lines).}
    \label{fig:blocks}
\end{figure}


We also have
$$
\frac{|\Delta_{i_{j-1}}|}{|\Delta_{i_j-1}|}=
\prod_{x=i_{j-1}}^{i_j-2}\frac{|\Delta_x|}{|\Delta_{x+1}|}\ge 
\prod_{x=i_{j-1}}^{i_j-2}2=2^{i_j-i_{j-1}-1}.$$
As $j>1$, we have $[G:G_{(\Sigma_{i_{j-1}})}]\ge 2$ and hence $$[G:
G_{(\Sigma_{i_{j-1}})}\cap G_{\{\Delta_{i_j-1}\}}
]\ge 2^{i_{j}-i_{j-1}}.$$ From~\eqref{aux} and~\eqref{thursday}, we get
$i_j-i_{j-1}\le \log_2(f(h(c,f_{j-1}(c)))!)$ and therefore
\begin{align*}
i_j&\le i_{j-1}+\log_2(f(h(c,f_{j-1}(c)))!)\\&\le f_{j-1}(c)+\log_2(f(h(c,f_{j-1}(c)))!)
=f_j(c),
\end{align*}
where the last equality follows from~\eqref{functions}.

From above, we have $\ell=i_\kappa\le f_\kappa(c)$. Moreover, from~\eqref{eq_1}, we have $\ell\le f_{\lceil\log_2(c)+1\rceil}(c)$ and hence, from
Theorem~\ref{thrm:1}, we deduce
$n=|\Omega|\le h(c,f_{\lceil\log_2(c)+1\rceil}(c))=F(c)$.

\section{Proof of Theorem~\ref{thrm:1}}
In this section we prove Theorem~\ref{thrm:1}, assuming the veracity of the following lemma: we postpone the proof of this lemma for not breaking the flow of the argument.

\begin{lemma}\label{lem:aux1}
There exists a function $\theta:\mathbb{N}\to\mathbb{N}$ with the property that,
if $T$ is a non-abelian simple group and  $M$ is a proper subgroup of $T$, then $T$ contains a subset $C$ such that
\begin{enumerate}
\item\label{eqlem:0}$|C|\ge 2$, 
\item\label{eqlem:1} $\{yx^{-1}\mid x,y\in C,x\ne y\}\cap M^\varphi=\emptyset$, for every $\varphi\in \mathrm{Aut}(T)$, 
\item\label{eqlem:2} $|T|\le \theta(|C|)$.
\end{enumerate}
\end{lemma}

 Let $G$ be a transitive permutation group of degree $n$ with domain $\Omega$. Assume that the derangement graph of $G$ contains no clique of size $c$, and let $\ell$ be the maximum length of a normal imprimitivity series for $G$. We argue by induction on $\ell$.

If $\ell=1$, then every non-trivial normal subgroup of $G$ is transitive on $\Omega$, otherwise we contradict the maximality of $\ell$. Hence $G$ is quasiprimitive. In this case, the result follows from~\cite[Theorem~1.1]{FPS}, and we may take $h(c,1)$ to be the function denoted by $f_1$ in that theorem. 

Now assume that $\ell>1$ and that $h(c,\ell-1)$ is already defined, with the meaning given in the statement of Theorem~\ref{thrm:1}.  
Set
\begin{align}\label{def:hcell}
h(c,\ell)=\max&\left\{{c\choose 2}\,h(c,\ell-1)^2,h(c,\ell-1)\theta(c)^{\log_2(c)h(c,\ell-1)!},\right.\\\nonumber
& h(c,\ell-1)c^{h(c,\ell-1)!2^{h(c,\ell-1)}}, h(c,\ell-1)|\mathbb{M}|^{c(h(c,\ell-1)!)^2},\\\nonumber
&\left. h(c,\ell-1)(f_{Lie}((ch(c,\ell-1)!)^2))^{c(h(c,\ell-1))^2}\right\},\nonumber
\end{align}
where $\theta$ is the is the function defined in Lemma~\ref{lem:aux1}, $\mathbb{M}$ is the Monster group and $f_{Lie}$ is the function defined in Theorem~\ref{thrm:4}.

 Let
$$
\Omega=\Sigma_\ell<\Sigma_{\ell-1}<\cdots<\Sigma_1<\Sigma_0=\{\Omega\}
$$
be a normal imprimitivity series of length $\ell$, and let $K=G_{(\Sigma_{\ell-1})}$ be the kernel of the action of $G$ on $\Sigma_{\ell-1}$.

Note that $K\neq 1$ because $\Sigma_{\ell-1}$ is a non-trivial normal system of imprimitivity, and let $N$ be a minimal normal subgroup of $G$ contained in $K$. By the maximality of $\ell$, we deduce that $K$ and $N$ have the same orbits on $\Omega$. That is,
\begin{equation}\label{eq:25}
K=NK_\omega,\qquad \forall\omega\in \Omega.
\end{equation}

\smallskip

\noindent\textsc{Assume that $N$ is abelian.}

\smallskip

 Set 
$$
\mathcal{S}=\bigcup_{\omega\in \Omega}N_\omega.
$$
For each $\Delta\in \Sigma_{\ell-1}$, choose a point $\omega_\Delta\in \Delta$. Since $N$ is abelian and transitive on $\Delta$, the stabilizer $N_{\omega_\Delta}$ fixes $\Delta$ pointwise. Therefore,
$$
\mathcal{S}=\bigcup_{\Delta\in \Sigma_{\ell-1}} N_{\omega_\Delta}.
$$
This implies
\begin{align}\label{lauree}
|\mathcal{S}| &\le |\Sigma_{\ell-1}|\, |N_{\omega_\Delta}| 
               \le h(c,\ell-1)\,|N_{\omega_\Delta}|,
\end{align}
where in the last inequality we apply the inductive hypothesis to the group $G/G_{(\Sigma_{\ell-1})}$ acting on $\Sigma_{\ell-1}$.

Since $G$ has no clique of size $c$ in its action on $\Omega$, the same holds for $N$. Thus, for every $(x_1,\ldots,x_c)\in N^c$, there exist indices $i,j\in \{1,\ldots,c\}$ with $i<j$ such that $x_jx_i^{-1}$ fixes some point of $\Omega$, that is, $x_jx_i^{-1}\in \mathcal{S}$ and hence $x_j\in \mathcal{S}x_i$. Therefore, in the $c$-tuple $(x_1,\ldots,x_c)$, the $j$-th coordinate differs from the $i$-th coordinate by an element of $\mathcal{S}$. Consequently,
\begin{align*}
|N|^c \le {c\choose 2} |N|^{c-1} |\mathcal{S}|.
\end{align*}
Combining this with~\eqref{lauree}, we obtain
$$
|\Delta|=|N:N_{\omega_\Delta}|\le {c\choose 2 } h(c,\ell-1).
$$
Therefore, using~\eqref{def:hcell}, we have
$$
n=|\Omega| = |\Sigma_{\ell-1}|\,|\Delta| \le {c\choose 2} h(c,\ell-1)^2\le h(c,\ell).
$$
This completes the inductive step in the case $N$ is abelian.

\smallskip

\noindent\textsc{Assume that $N$ is non-abelian.}

\smallskip

We start by setting some notation. We have $N=T^s$, for some non-abelian simple group $T$ and for some positive integer $s$. We write $N=T_0\times \cdots \times T_{s-1}$, where $T_0,\ldots,T_{s-1}$ are the simple direct factors of $N$. As $N\unlhd K=G_{(\Sigma_{\ell-1})}$, the group $N$ is the direct product of minimal normal subgroups of $K$, that is,
$$N=N_1\times \cdots\times N_\kappa,$$
where $N_j$ is a minimal normal subgroup of $K$, for every $j\in \{1,\ldots,\kappa\}$. As $N$ is minimal normal in $G$ and $K\unlhd G$, the action by conjugation of $G$ on $\{N_1,\ldots,N_\kappa\}$ is transitive. In particular, $\kappa$ divides $s$ and we may write $s=r\kappa$, for some $r\in\mathbb{N}$. Up to relabeling the indexed set, we may assume that $N_j=T_{(j-1)r}\times T_{(j-1)r+1}\times \cdots\times T_{(j-1)r+r-1}$. Thus $N_j\cong T^r$.

As $G/K$ acts transitively by conjugation on $\{N_1,\ldots,N_\kappa\}$, we have $\kappa\le [G:K]=[G:G_{(\Sigma_{\ell-1})}]\le |\Sigma_{\ell-1}|!$ and hence
\begin{equation}\label{eq:boundkappa}
\kappa\le h(c,\ell-1)!.
\end{equation}

 We let $\pi_i:N\to T_i$ be the natural projection onto the $i$th coordinate. Let $\alpha\in \Omega$. We now distinguish two cases, depending on whether there exists $i\in \{1,\ldots,s\}$ with $\pi_i(N_\alpha)<T_i$, or $\pi_i(N_\alpha)=T_i$ for every $i$.

\smallskip

\noindent\textsc{Case 1: }$\pi_i(N_\alpha)$ is a proper subgroup of $T_i$, for some $i\in \{0,\ldots,s-1\}$.
\smallskip

Up to relabeling the index set $\{0,\ldots,\ell-1\}$,  we may suppose that $i=0$. Let $M$ be a maximal subgroup of $T_0$ with $\pi_0(N_\alpha)\le M$.

From~\eqref{eq:25}, we have $K=NK_\alpha$ and hence $K_\alpha$ acts transitively by conjugation on $\{T_0,\ldots,T_{r-1}\}$. Moreover, as $N_\alpha\unlhd K_\alpha$, we deduce that, for every $i\in \{0,\ldots,r-1\}$, there exists $\varphi_i\in\mathrm{Aut}(T)$ with $\pi_i(N_\alpha)\le M^{\varphi_i}$.

Let $\theta$ and $C$ be the function and the subset arising in Lemma~\ref{lem:aux1}. For each $i\in \{1,\ldots,\kappa\}$, let $C_i=C^r$ be the subgroup of $N_i$ consisting of the Cartesian product of $r$ copies of $C$. Then let $$C'=\mathrm{Diag}(C_1\times C_2\times\cdots\times C_\kappa)\cong C^r.$$
 The condition~\eqref{eqlem:1} in Lemma~\ref{lem:aux1}  implies that $C'$ is a clique for the derangement graph of $G$ in its action on $\Omega$. Therefore, $|C|^r=|C'|< c$. From Lemma~\ref{lem:aux1}~\eqref{eqlem:0}, we have $r\le \log_2(c)$. Moreover, from Lemma~\ref{lem:aux1}~\eqref{eqlem:2}, we have $|T|\le \theta(c)$. Thus $$|N|=|T|^{r\kappa}\le \theta(c)^{\log_2(c)\kappa}\le\theta(c)^{\log_2(c)h(c,\ell-1)!}.$$ This and~\eqref{def:hcell} imply
\begin{align*}
|\Omega|&=|\Sigma_{\ell-1}||\alpha^N|\le |\Sigma_{\ell-1}||N|\\
&\le h(c,\ell-1)\theta(c)^{\log_2(c)h(c,\ell-1)!}\le h(c,\ell).
\end{align*}
This completes the inductive step of Case~1.

\smallskip

\noindent\textsc{Case 2: }$\pi_i(N_\alpha)=T_i$, for every $i\in \{1,\ldots,s\}$.

\smallskip

Given a positive integer $a$, we denote by $\mathrm{Diag}(T^a)$ any subgroup of the direct product $T^a$ of the form
\[
\{(t^{\varphi_1},t^{\varphi_2},\ldots,t^{\varphi_a}) \mid t\in T\},
\]
where $\varphi_1,\ldots,\varphi_a\in \mathrm{Aut}(T)$. 
This notation is slightly abusive, since $\mathrm{Diag}(T^a)$ depends on the choice of the automorphisms $\varphi_1,\ldots,\varphi_a$, but this will not cause any ambiguity in what follows.

By Scott's Lemma, there exists a partition $\pi$ of the set $\{0,\ldots,s-1\}$ such that
\begin{equation}\label{eq:Nalpha}
N_\alpha=\prod_{P\in \pi}\mathrm{Diag}\left(\prod_{x\in P}T_x\right).
\end{equation}

Each element $g\in G$ induces a permutation of $\{0,\ldots,s-1\}$ defined by $i^g=i'$ whenever $T_i^g=T_{i'}$. 
Since $N_\alpha\unlhd G_\alpha$, it follows that $G_\alpha$ preserves the partition $\pi$, that is, $P^g\in \pi$ for every $P\in\pi$ and every $g\in G_\alpha$. 
As $K=K_\alpha N$ and as $N$ acts trivially on $\{0,\ldots,s-1\}$, we deduce that $\pi$ is fixed by $G_\alpha K=G_{\{\Delta\}}$, where $\Delta\in\Sigma_{\ell-1}$ is the block containing $\alpha$. 
In particular, $\pi$ depends only on $\Delta$ and not on the specific choice of $\alpha\in\Delta$. 
We denote this partition by $\pi_\Delta$. 
Similarly, for each $\Lambda\in\Sigma_{\ell-1}$ we obtain a partition $\pi_{\Lambda}$ of $\{0,\ldots,s-1\}$. 
Thus we obtain a family of at most $|\Sigma_{\ell-1}|$ partitions of $\{0,\ldots,s-1\}$, on which $G$ acts transitively. Set
\begin{eqnarray*}\label{eq:defpartitions}
\Pi=\{\pi_\Lambda\mid\Lambda\in\Sigma_{\ell-1}\}.
\end{eqnarray*}
We also let $X=\{0,\ldots,s-1\}$ and $$X_j=\{(j-1)r,(j-1)r+1,\ldots,(j-1)r+r-1\},$$ for each $j\in \{1,\ldots,\kappa\}$. Moreover, let $\pi_{\Lambda,j}$ be the partition of $X_j$ obtained by intersecting each part of $\pi_{\Lambda}$ with  $X_j$ and keeping only the nonempty intersections. Since $K$ preserves $\pi_{\Lambda}$ and acts transitively on $X_j$, the partition $\pi_{\Lambda,j}$ must be uniform, that is, all parts of $\pi_{\Lambda,j}$ have the same cardinality. To clarify these definitions, see the example in Figure \ref{fig:ex_partition}.
\begin{figure}[H]
    \centering
    \includegraphics[width=12cm]{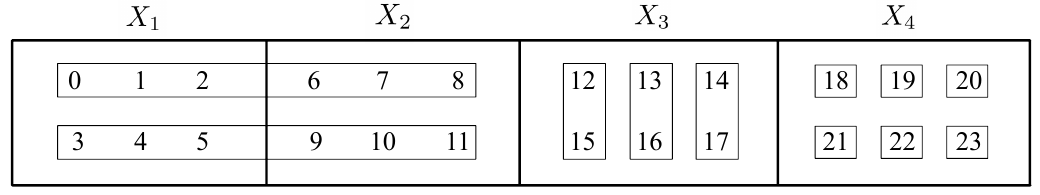}
    \caption{\footnotesize{\textbf{Example:} Here, $\kappa=4$, $r=6$, $\pi_{\Delta}$ is the partition of $\{0,\dots,23\}$ represented in the picture above. \[
\begin{aligned}
 & \pi_{\Delta,1}=\{\{0,1,2\},\{3,4,5\}\},
     & \pi_{\Delta,3}&=\{\{12,15\},\{13,16\},\{14,17\}\},\\
 & \pi_{\Delta,2}=\{\{6,7,8\},\{9,10,11\}\}, 
     & \pi_{\Delta,4}&=\{\{i\}\mid 18 \le i \le 23\}.
\end{aligned}
\]}}
    \label{fig:ex_partition}
\end{figure}
\noindent\textsc{Subcase 2A: }There exists $j\in \{1,\ldots,\kappa\}$, and $P\in \pi_\Delta$ with $|X_j\cap P|\ge 2$.

\smallskip 

 As  $G$ acts transitively on $\Pi$, for every $\pi_\Lambda\in \Pi$, there exists $j\in \{1,\ldots,\kappa\}$ and $P\in \pi_\Lambda$ with $|X_j\cap P|\ge 2$, that is,  $\pi_{\Lambda,j}$ has parts of size at least 2. And similarly, as  $G$ acts transitively on $\{X_1,\ldots,X_{\kappa}\}$, we deduce that for every $j\in \{1,\ldots,\kappa\}$, there exists $\pi_\Lambda \in \Pi$ and $P\in \pi_\Lambda$ with $|X_j\cap P|\ge 2$.

For $j\in \{1,\ldots,\kappa\}$, let $$\Pi_j=\{\pi_\Lambda\in\Pi\mid \hbox{there exists }P\in \pi_\Lambda \hbox{ with }|X_j\cap P|\ge 2\}.$$
From the previous paragraph, $\Pi=\bigcup_j \Pi_j$ and $G$ acts transitively on the set $\{\Pi_j\mid j\in \{1,\ldots,\kappa\}\}$.

Now, $\{\pi_{\Lambda,j}\mid \pi_\Lambda\in \Pi_j\}$ is a family of non-empty uniform partitions of $X_j$ where each part in $\pi_{\Lambda,j}$ has cardinality at least $2$. Therefore, from Lemma~\ref{lem:aux2}, there exists a subset $Y_j$ of $X_j$ of cardinality at least $r/2^{|\Sigma_{\ell-1}|}$ such that no part of $\pi_{\Lambda,j}$ is contained in $Y_j$, for every $\pi_{\Lambda,j}$. Moreover, as $G$ acts transitively on $\{\Pi_j\mid j\in \{1,\ldots,\kappa\}\}$, we may choose $Y_1,\ldots,Y_\kappa$ such that, for every $j,j'\in \{1,\ldots,\kappa\}$, there exists $g\in G$ with $Y_{j}^g=Y_{j'}$, in particular, $|Y_j|=|Y_{j'}|$ for every $j,j'\in \{1,\ldots,\kappa\}$.  For every $j\in \{2,\ldots,\kappa\}$, let $f_j:Y_1\to Y_j$ be a bijection.

Let $C$ be the subset of $N$ consisting of all elements $g=(t_i)_i$ such that
\begin{equation}\label{cases}
t_i=
\begin{cases}
1&\textrm{when }i\in X\setminus(Y_1\cup\cdots\cup Y_\kappa),\\
t_{f_j^{-1}(i)}&\textrm{when }i\in Y_j, j\ge 2.
\end{cases}
\end{equation}
Clearly, $C$ is a subgroup of $N$ and $C\cong T^{|Y_1|}$ because the elements of $C$ are uniquely determined by the coordinates in $Y_1$.
We claim that $C$ is a clique for the derangement graph of $G$ in its action on $\Omega$. Let $g\in C$ and write $g=(t_i)_i$, with $t_i\in T_i$. Let $Z=\{i\in \{0,\ldots,s-1\}\mid t_i=1\}$. If $g$ fixes some element, $\beta$ say, of $\Omega$, then from the description in~\eqref{eq:Nalpha} of the point stabilizers in $N$, we deduce that $Z$ is a union of parts of $\pi_\Lambda$, where $\Lambda$ is the element of $\Sigma_{\ell-1}$ with $\beta\in \Lambda$. As $\Pi=\bigcup_j\Pi_j$, $\pi_\Lambda\in \Pi_j$, for some $j\in \{1,\ldots,\kappa\}$. In particular, $Z\cap X_j$ is a union of parts of $\pi_{\Lambda,j}$. By~\eqref{cases}, $Z\cap X_j\supseteq X_j\setminus Y_j$ and by construction $Y_j$ contains no parts of $\pi_{\Lambda,j}$. Therefore, we have $Z\cap X_j=X_j$, that is, $X_j\subseteq Z$. In other words, $g$ is constantly equal to $1$ on the coordinates labeled by the elements in $X_j$. Now,~\eqref{cases} implies $g=1$. This has shown that $C$ is a clique.

Since, by hypothesis, there are no cliques of cardinality greater than $c$, we get
\[
|N|^{\frac{1}{\kappa 2^{|\Sigma_{\ell-1}|}}}=|T|^{\frac{r}{2^{|\Sigma_{\ell-1}|}}}
\leq |T|^{|Y_1|} = |C|
\leq c.
\]
From~\eqref{eq:boundkappa}, it follows that
\begin{align*}
|N|&\leq c^{\kappa 2^{\,h(c,\ell-1)}}\le  c^{h(c,\ell-1)!2^{h(c,\ell-1)}}.
\end{align*}
Consequently,
\[
|\Omega|
=|\Sigma_{\ell-1}||\alpha^N|
\leq |\Sigma_{\ell-1}||N|
\leq h(c,\ell-1)\,c^{h(c,\ell-1)!2^{h(c,\ell-1)}}
\leq h(c,\ell).
\]
This completes the inductive step in Subcase~2A.

\smallskip

\noindent\textsc{Subcase 2B: }For every $j\in \{1,\ldots,\kappa\}$ and $P\in \pi_\Delta$, we have $|X_j\cap P|\le 1$.

\smallskip 
We claim that $r\le \kappa c$ and hence, from~\eqref{eq:boundkappa},
\begin{equation}\label{eq:boundr}
r\le ch(c,\ell-1)!.
\end{equation}
We argue by contradiction and suppose that $r>\kappa c$. Let $x\in T$ with $x\ne 1$. Using $x$, we construct $c$ elements of $N$ forming a clique of size $c$, leading to a contradiction.

For each $j\in \{1,\ldots,\kappa\}$, let 
$X_j=X_{j,1}\cup \cdots \cup X_{j,c}$ 
be a partition of $X_j$, where 
$|X_{j,1}|=\cdots =|X_{j,c-1}|=\kappa$ and 
$|X_{j,c}|=r-(c-1)\kappa$. 
Observe that $|X_{j,c}|> \kappa$.

For each $\ell\in \{1,\ldots,c\}$ and $j\in \{1,\ldots,\kappa\}$, let $n_{j,\ell}\in N_{j}$ be an element which is identically equal to $1$ on the coordinates labelled by the elements of $X_j\setminus X_{j,\ell}$, and on the coordinates labelled by the elements of $X_{j,\ell}$ has exactly $j$ coordinates equal to $x$, the remaining coordinates being equal to $1$.  
Now define  
$g_\ell=n_{1,\ell}\cdots n_{\kappa,\ell}$  
and let $C=\{g_1,\ldots,g_c\}$. Since the construction of these elements is quite involved, before proceeding with the proof, we give an explicit example.

\begin{example}
Consider the set
\begin{align*}
X_1=&\{0,1,\dots,r-1\}\\
=&X_{1,1}\cup X_{1,2} \cup \cdots \cup X_{1,c}\\
=&\{0,1,\dots,\kappa-1\}\cup \{\kappa,\dots,2k-1\}\cup \cdots\cup \{(c-1)\kappa,\dots,r-1\}.
\end{align*}
Then, we define the element 
$$
n_{1,1}=
\bigl(
\underbrace{x,1,\dots,1}_{X_{1,1}}
\mid
\underbrace{1,\dots,1}_{X_{1,2}}
\mid
\underbrace{1,\dots,1}_{X_{1,3}}
\mid
\cdots
\mid
\underbrace{1,\dots,1}_{X_{1,c}}
\bigr) \in N_1,
$$
where the verticals bars partition the coordinates of $N_1$ in the sets $X_{1,\ell}$ defined as above.
Partitioning in the same way all the other sets $X_2,\dots,X_{\kappa}$, we can define  the elements $n_{j,1}\in N_j$ as
\begin{align*}
    n_{1,1}= 
\bigl(x,1,\dots,1&\mid1,\dots,1\mid 1,\dots,1 \mid \cdots \mid 1,\dots,1\bigr),\\
    n_{2,1}=
\bigl(x,x,1,\dots,1 &\mid 1,\dots,1 \mid
1,\dots,1 \mid \cdots \mid 1,\dots,1 \bigr),\\
 n_{3,1}= 
\bigl( x,x,x,1\dots,1 &\mid 1,\dots,1 \mid
1,\dots,1 \mid \cdots \mid 1,\dots,1 \bigr),\\
&\cdots
\end{align*}
Let now $n_{j,2}\in N_j$ be defined as
\begin{align*}
    n_{1,2}= 
\bigl(\underbrace{1,\dots,1}_{X_{1,1}}
\mid
\underbrace{x,1,\dots,1}_{X_{1,2}}
&\mid
\underbrace{1,\dots,1}_{X_{1,3}}
\mid
\cdots
\mid
\underbrace{1,\dots,1}_{X_{1,c}}
\bigr),\\
    n_{2,2}=
\bigl(
1,\dots,1
\mid
x,x,1,\dots,1
&\mid
1,\dots,1
\mid
\cdots
\mid
1,\dots,1
\bigr),\\
 n_{3,2}= 
\bigl(
1,\dots,1
\mid
x,x,x,1,\dots,1
&\mid
1,\dots,1
\mid
\cdots
\mid
1,\dots,1
\bigr) ,\\
&\cdots
\end{align*}
Then, for example, 
\[
\begin{aligned}
g_1
= {}(& 
\underbrace{x,1,\dots,1 \mid 1,1,\dots,1\mid \cdots \mid 1,\dots,1}_{n_{1,1}},\,
n_{2,1},\, \dots,\, n_{k,1}
),\\
g_2
= {}(& 
\underbrace{1,1,\dots,1 \mid x,1, \dots, 1 \mid \cdots \mid 1,\dots,1}_{n_{1,2}},\,
n_{2,2},\, \dots,\, n_{k,2}
).
\end{aligned}
\]
\end{example}
We claim that $C$ is a clique in the derangement graph of $G$ in its action on $\Omega$. 
For every $i,j\in \{1,\ldots,c\}$ with $i<j$, the element $g=g_jg_i^{-1}$ has, by construction, $2$ entries different from $1$ on the coordinates labelled by the elements of $X_1$, $4$ entries different from $1$ on the coordinates labelled by the elements of $X_2$, and, in general, $2y$ entries different from $1$ on the coordinates labelled by the elements of $X_y$, for every $y\in \{1,\ldots,\kappa\}$.

Suppose that $g$ fixes some element $\beta\in \Omega$. Then, by~\eqref{eq:Nalpha}, the set of coordinates on which $g$ is equal to $1$ is a union of parts of $\pi_\Lambda$, where $\Lambda\in \Sigma_{\ell-1}$ and $\beta\in \Lambda$. Since $N$ acts faithfully on $\Omega$, the partition $\pi_\Lambda$ is not trivial, and hence there exists $P\in \pi_\Lambda$ with $|P|\ge 2$. 

Since $|X_j\cap P|\le 1$ for every $j\in \{1,\ldots,\kappa\}$, there exist two distinct indices $j,j'\in \{1,\ldots,\kappa\}$ such that 
$|X_j\cap P|=|X_{j'}\cap P|=1$.

As $K$ acts transitively on $X_j$ and on $X_{j'}$, and as $K$ fixes $\pi_\Lambda$, we deduce that every part of $\pi_\Lambda$ intersecting $X_j$ in one point also intersects $X_{j'}$ in one point. Therefore, the elements of $N_\beta$ have the same number of coordinates equal to $1$ on the coordinates labelled by the elements of $X_j$ as on those labelled by the elements of $X_{j'}$. 

However, as observed above, $g$ has a different number of coordinates equal to $1$ on these two sets. Hence $g$ does not fix $\beta$. This establishes~\eqref{eq:boundr}.

Now we divide our proof depending on the simple group $T$. Assume that $T$ is a sporadic simple group and let $\mathbb{M}$ denote the Monster group. Then 
$$|N|\le |T|^{r\kappa}\le |\mathbb{M}|^{c(h(c,\ell-1)!)^2}$$
and hence
\[
|\Omega|
=|\Sigma_{\ell-1}||\alpha^N|
\leq |\Sigma_{\ell-1}||N|
\leq h(c,\ell-1)|\mathbb{M}|^{c(h(c,\ell-1)!)^2}
\leq h(c,\ell).
\]
This completes the inductive step in Subcase~2B when $T$ is a sporadic simple group.

Assume that $T$ is a simple group of Lie type. We claim that 
\begin{equation}\label{bound:mpr}
\mathrm{mpr}^\ast(T)<sc.
\end{equation}
We argue by contradiction, and suppose that $\mathrm{mpr}^\ast(T)\ge sc$. Let $x\in T$ be an element of prime order $p$ such that $\mathrm{mpr}^\ast(x)=\mathrm{mpr}^\ast(T)$. 
By definition, the cyclic group $\langle x\rangle$ contains $\mathrm{mpr}^\ast(T)$ elements lying in pairwise distinct $\mathrm{Aut}(T)$-classes. Let $x^{e_1},\ldots,x^{e_s}$ be $s$ of these elements, and let 
$g=(t_i)_i\in N$ with $t_i=x^{e_i}$. 
Since all coordinates of $g$ belong to distinct $\mathrm{Aut}(T)$-classes, it follows from~\eqref{eq:Nalpha} that $g$ fixes no point of $\Omega$. As $g$ has prime order, $C=\langle g\rangle$ is a clique in the derangement graph of $G$ in its action on $\Omega$. Hence $p=|C|<c$. However, this yields a contradiction, since 
$p-1 \ge \mathrm{mpr}^\ast(x)\ge sc\ge c$.

From~\eqref{eq:boundkappa},~\eqref{eq:boundr},~\eqref{bound:mpr} and Theorem~\ref{thrm:4}, we obtain
$$|T|\le f_{\mathrm{Lie}}(sc)\le f_{\mathrm{Lie}}((ch(c,\ell-1)!)^2).$$
Therefore,
$$|\Omega|\le h(c,\ell-1)\,\bigl(f_{\mathrm{Lie}}((ch(c,\ell-1)!)^2)\bigr)^{c(h(c,\ell-1))^2}\le h(c,\ell).$$

Finally, assume that $T=\mathrm{Alt}(m)$ is the alternating group of degree $m\ge 5$. We claim that
\begin{equation}\label{eq:alt}
m<2^sc.
\end{equation}
We argue by contradiction, and suppose that $m\ge 2^sc$. By Bertrand's postulate, for every $i\in \{1,\ldots,s\}$ there exists a prime $p_i$ in the open interval $(2^ic,2^{i+1}c)\subseteq \mathbb{R}$. 

Now let $g=(t_i)_i\in N$, where each $t_i$ is an element of $T$ of prime order $p_i$, and let
$C=\{g^{\ell}\mid \ell\in \{0,\ldots,p_1-1\}\}$.
For every $i,j\in \{0,\ldots,p_1-1\}$ with $j>i$, the element $g^{j}g^{-i}=g^{j-i}$ fixes no point of $\Omega$, since its $s$ coordinates lie in pairwise distinct $\mathrm{Aut}(T)$-classes. Therefore, $C$ is a clique for the derangement graph of $G$ in its action on $\Omega$. Hence
$c>|C|\ge p_1>c$,
which is a contradiction. This proves~\eqref{eq:alt}.

From~\eqref{eq:alt} we obtain $|T|\le (2^sc)!$, and hence, using~\eqref{eq:boundkappa} and~\eqref{eq:boundr}, we deduce as before that
\begin{align*}
|\Omega|&\le |T|^s|\Sigma_{\ell-1}|
\le h(c,\ell-1)((2^sc)!)^s\\
&\le h(c,\ell-1)((2^{c(h(c,\ell-1)!)^2}c)!)^{c(h(c,\ell-1))^2}
\le h(c,\ell).
\end{align*}
This completes the inductive step in Subcase~2B.

\section{Proof of Lemma~\ref{lem:aux1}}\label{sec:lemma}
The proof of Lemma~\ref{lem:aux1} draws on several ideas from~\cite{FPS} together with Theorem~\ref{thrm:4LPS}.  
Specifically, Theorem~\ref{thrm:4LPS} provides, for each simple group of Lie type $T$ and each maximal subgroup $M$, three explicitly described primes with the property that, aside from a list of exceptional 
cases (in~\cite[Tables~$10.1-10.6$]{LPS}), at least one of these primes does not divide $|M|$.  
In most instances (though not always), these primes are of the form $q_\ell$, where the parameters $q$ and $q_\ell$ are defined as in Section~\ref{sec:auxiliary}, recall $q_{\ell}=P[\Phi_{\ell e}(p)]$.

Note that the set $C$  in the statement of Lemma~\ref{lem:aux1} has a combinatorial interpretation:  for every $\varphi\in\mathrm{Aut}(T)$, $C$ is a clique in the derangement graph of $T$ in its action on the right cosets of $M^\varphi$.

\begin{proof}Replacing $M$ if necessary, we may assume that $M$ is a maximal subgroup of $T$. From Theorem~\ref{thrm:saxl}, there exists $g \in T \setminus \bigcup_{\varphi \in \mathrm{Aut}(T)} M^\varphi$.  
Let $C=\{1,g\}$. Clearly $C$ satisfies~\eqref{eqlem:0} and~\eqref{eqlem:1}.  
This shows that we can always choose a set $C$ fulfilling the first two requirements, and therefore, for the remainder of the proof, it suffices to show that one may select $C$ so that $|T|$ is bounded above by a function of $|C|$. We use the CFSG. 

\smallskip

\noindent\textsc{Assume  $|T|\le |\mathbb{M}|$,} where $\mathbb{M}$ denotes the Monster.  
\smallskip

In this case,~\eqref{eqlem:2} is satisfied provided we choose $\theta$ so that $\theta(2)\ge |\mathbb{M}|$.  
Thus, for the remainder of the proof, we may assume that $|T|>|\mathbb{M}|$.

\smallskip
\noindent\textsc{Assume the $\mathrm{Aut}(T)$-conjugacy class of $M$ equals the $T$-conjugacy class of $M$.}

\smallskip

Let $C$ be a clique of maximum size in the derangement graph of $T$ acting on the cosets of $M$.  
Because the $\mathrm{Aut}(T)$–conjugacy class of $M$ coincides with its $T$–conjugacy class, $C$ satisfies~\eqref{eqlem:1}.  
Moreover, by Theorem~\ref{thrm:FPS}, $C$ also satisfies~\eqref{eqlem:2} provided we choose $\theta$ such that $\theta(|C|)\ge f_1(|C|+1)!$.  
Therefore, for the remainder of the proof we may suppose that the $\mathrm{Aut}(T)$-conjugacy class of $M$ splits into at least two $T$-conjugacy classes.

\smallskip

\noindent\textsc{Assume $T$ is the alternating group $\mathrm{Alt}(m)$}, for some $m\ge 5$.

\smallskip
Since $|T|=m!/2>|\mathbb{M}|$, we have $m\neq 6$ and therefore $\mathrm{Aut}(T)=\mathrm{Sym}(m)$.  
Assume first that $M$ is either intransitive or imprimitive on $\{1,\ldots,m\}$.  
Thus $M$ is one of the following:
\begin{itemize}
    \item $M=\mathrm{Alt}(m)\cap (\mathrm{Sym}(k)\times\mathrm{Sym}(m-k))$ for some $k$ with $1\le k<m/2$, or
    \item $M=\mathrm{Alt}(m)\cap \mathrm{Sym}(k)\mathrm{wr}\mathrm{Sym}(m/k)$ for some divisor $k$ of $m$ with $1<k<m$.
\end{itemize}
In either case, ${\bf N}_{\mathrm{Sym}(m)}(M)\nleq \mathrm{Alt}(m)$. Because ${\bf N}_{\mathrm{Sym}(m)}(M)\nleq \mathrm{Alt}(m)$, the $\mathrm{Aut}(T)$–conjugacy class of $M$ coincides with its $T$–conjugacy class, contradicting our earlier assumption.

Now assume that $M$ is primitive on $\{1,\ldots,m\}$.  
Since $m\ge 8$, Bertrand's postulate~\cite[p.~498]{HW} guarantees a prime $p$ with $m/2<p\le m-3$.  
If $p$ divides $|M|$, then $M$ contains a $p$–cycle in its action on $\{1,\ldots,m\}$; by Jordan's theorem~\cite[Theorem~3.3E]{dixon}, this forces $M=\mathrm{Alt}(m)$, contradicting $M<T$.  
Thus $p$ does not divide $|M|$.  
If we let $C$ be a cyclic subgroup of $T=\mathrm{Alt}(m)$ of order $p$, then $C$ satisfies~\eqref{eqlem:0} and~\eqref{eqlem:1}, and it satisfies~\eqref{eqlem:2} provided $\theta$ is chosen so that $\theta(p)\ge m!/2$, for every prime $p$ with $m/2<p\le m-3$.

\smallskip

\noindent\textsc{Assume $T$ is a simple group of Lie type.}

\smallskip
 Let $q$ be the prime power, depending on the type of $T$, defined before Lemma~\ref{lem:cats}. Thus $q=p^{e}$ for some prime $p$ and some positive integer $e$. For each $\ell\in\mathbb{N}$, set $q_{\ell}=P[\Phi_{\ell e}(p)]$.  

Since $|T|>|\mathbb{M}|$, we may discard from consideration all groups $T$ listed in~\cite[Table~10.4]{LPS}, as well as those appearing in~\cite[Table~10.1, 10.2, 10.3, 10.5]{LPS} whose order is at most $|\mathbb{M}|$.  
Furthermore, every simple group occurring in~\cite[Table~10.2]{LPS} also appears in~\cite[Table~10.1]{LPS} (the two tables differ only in the choice of primes), so we may exclude the groups from~\cite[Table~10.2]{LPS} as well.

Assume first that every prime listed in the second column of~\cite[Tables~10.1, 10.3, 10.5]{LPS} is of the form $q_{\ell}$ for some~$\ell$. Suppose further that no option for $M$ appears in the fifth column of these tables. Then $|M|$ is not divisible by $q_{\ell}$ for some~$\ell$. Let $C$ be a cyclic subgroup of $T$ of order $q_{\ell}$. Then $C$ satisfies~\eqref{eqlem:0} and~\eqref{eqlem:1}. 

Consulting~\cite[Tables~10.1, 10.3, 10.5]{LPS}, we observe that in each case the Lie rank of $T$ is bounded above by a function of~$\ell$, and since $\ell$ divides $q_{\ell}-1$, this bound is ultimately a function of $q_{\ell}$. Therefore, by Lemma~\ref{lem:cats}, the subgroup $C$ also satisfies~\eqref{eqlem:2}.

Now assume that not all primes in the second column of~\cite[Tables~10.1, 10.3, 10.5]{LPS} are of the form $q_{\ell}$, or that they are all of this form but some option for $M$ appears in the fifth column. Then one of the following holds:
\begin{enumerate}
\item\label{case1} $T=\mathrm{PSL}_d(q)$ with $d\le 4$,
\item\label{case2} $T=\mathrm{PSU}_d(q)$ with $d\le 4$,
\item\label{case3} $T=\mathrm{PSp}_4(q)$,
\item\label{case4} $T={}^2B_2(2^{2e+1})$ and $q_4$ divides $|M|$,
\item\label{case5} $T={}^2G_2(3^{2e+1})$ and $q_6$ divides $|M|$,
\item\label{case6} $T=G_2(q)$ with $q>2$, $M\cong\mathrm{PSL}_2(13)$, and $\{q_3,q_6\}=\{7,13\}$,
\item\label{case7} $(T,M)$ appears in Table~\ref{table:our}.
\end{enumerate}

\begin{table}[!ht]
\begin{tabular}{cccc}\hline
Line&$T$ & Conditions& Possible $M$\\\hline
1&$\mathrm{PSp}_{2m}(q)$&$m\ge 4$ even, $q$ even&$\Omega_{2m}^-(q)\unlhd M$\\
2&$\mathrm{P}\Omega_{2m+1}(q)$&$m=4k\ge 8$&$\Omega_{2m}^-(q)\unlhd M$,\\
3&                            &           &$2^{2m}\cdot A_{2m+1}\unlhd M$,\\
4&                            &           &$A_{2m+\delta}\unlhd M$ $(1\le \delta\le 3)$\\
5&$\mathrm{P}\Omega_{2m}^-(q)$&$m=4k\ge 8$&$A_{2m+\delta}\unlhd M$ $(1\le \delta\le 2) $\\
6&$\mathrm{P}\Omega_{2m}^+(q)$&$m$ even &$\Omega_{2m-1}(q)\unlhd M$\\\hline                           
\end{tabular}
\caption{Exceptional cases arising from the tables in~\cite{LPS}}\label{table:our}
\end{table}
We look at each case in turn. Observe that in Cases~\eqref{case1}--\eqref{case6}, since the Lie rank is bounded above by $4$, we only need to bound $q$ above with a function of $|C|$. 

Assume that~\eqref{case1} occurs with $d=2$.  
Let $r = P[p^{2e}-1]$.  
By Lemma~\ref{siegeltheorem} applied to the polynomial $x^{2e}-1$, there exists a constant $a$ such that whenever $q \ge a$, we have $r>5$.  
If $q \le a$, then $|T| \le a^3$, and hence the lemma follows immediately, since in this case $|T|$ is bounded above by an absolute constant.  
Thus we may assume $q \ge a$.
If $r$ is relatively prime to $|M|$, then the desired conclusion follows by taking a Sylow $r$-subgroup $C$ of $T$ and applying Lemma~\ref{siegeltheorem}.  
Suppose instead that $r$ divides $|M|$.  
Consulting the list of maximal subgroups of $T$ in~\cite[Tables~8.1-8.2]{bray}, and using the fact that $r>5$, we obtain the following possibilities:
\begin{itemize}
\item if $r \mid q-1$, then either $M \cong E_q : (q-1)/\gcd(q-1,2)$ or $M \cong D_{2(q-1)/\gcd(q-1,2)}$;
\item if $r \mid q+1$, then $M \cong D_{2(q+1)/\gcd(q-1,2)}$.
\end{itemize}
In every case,~\cite[Table~8.1, Column~Stab]{bray} shows that the $T$-conjugacy class of $M$ coincides with its $\mathrm{Aut}(T)$-conjugacy class.  
Therefore this situation falls under a case already eliminated by our earlier considerations, and may be disregarded.

Assume that either~\eqref{case1} occurs with $d\in \{3,4\}$, or~\eqref{case2} occurs. Let $\delta=1$ when $T=\mathrm{PSL}_d(q)$ and let $\delta=2$ when $T=\mathrm{PSU}_d(q)$. Let $q_{3\delta} = P[\Phi_{3\delta e}(p)]$. 
By Lemma~\ref{siegeltheorem} applied to the polynomial $\Phi_{3\delta e}(x)$, there exists a constant $a$ such that whenever $q \ge a$, we have $q_{3\delta}>7$.  
If $q \le a$, then $|T| \le a^{16\delta }$, and hence the lemma follows immediately, since in this case $|T|$ is bounded above by an absolute constant.  
Thus we may assume $q \ge a$.
If $q_{3\delta}$ is relatively prime to $|M|$, then the desired conclusion follows by taking a Sylow $q_{3\delta}$-subgroup $C$ of $T$ and applying Lemma~\ref{lem:cats}.  
Suppose instead that $q_{3\delta}$ divides $|M|$.  
Consulting the list of maximal subgroups of $T$ in~\cite[Tables~8.1 and~8.2]{bray}, and using the fact that $q_{3\delta}>7$ (to exclude the subgroups in the Aschbacher class $\mathcal{S}$), we obtain the following possibilities:
\begin{itemize}
\item $T=\mathrm{PSL}_3(q)$, $M\in\mathcal{C}_3$ is of type $(q^2+q+1):3$,
\item $T=\mathrm{PSU}_3(q)$, $M\in\mathcal{C}_3$ is of type $(q^2-q+1):3$,
\item $T=\mathrm{PSL}_4(q)$, $M\in\mathcal{C}_1$ is of type $E_q^3:\mathrm{GL}_3(q)$,
\item $T=\mathrm{PSU}_4(q)$, $M\in\mathcal{C}_1$ is of type $\mathrm{GU}_3(q)$.
\end{itemize}
Except when $T=\mathrm{PSL}_4(q)$,~\cite[Tables~8.3--8.11, Column~Stab]{bray} shows that the $T$-conjugacy class of $M$ coincides with its $\mathrm{Aut}(T)$-conjugacy class.  
Therefore this situation falls under a case already eliminated by our earlier considerations, and may be disregarded. Finally, when $T=\mathrm{PSL}_4(q)$ and $M$ is of type $E_q^3:\mathrm{GL}_3(q)$, we may repeat this argument replacing $q_{3\delta}$ with $q_{4}=P[\Phi_{4e}(p)]$.

Assume that~\eqref{case3} occurs.   If $|M|$ is relatively prime to $q_4$, then the conclusion follows immediately from Lemma~\ref{lem:cats}.  
Thus we may assume that $q_4$ divides $|M|$.  
If $q_4\le 7$, then Lemma~\ref{siegeltheorem} shows that $q$ is bounded above by an absolute constant, and therefore so is $|T|$.  
Hence we may further assume $q_4>7$.  
Consulting~\cite[Tables~8.12--8.15]{bray}, we see that in this situation either $M\in\mathcal{C}_3$ is of type $\mathrm{Sp}_2(q^2):2$, or 
$M\cong {}^2B_2(q)$,
where in the latter case $q$ is even and $e\ge 3$ is odd.  
If $q$ is odd, or if $q$ is even and $M\cong {}^2B_2(q)$, then from~\cite[Tables~8.12 and~8.14, Column~Stab]{bray} we deduce that the $\mathrm{Aut}(T)$–conjugacy class of $M$ coincides with its $T$–conjugacy class.  This contradicts one of our assumptions. Assume now that $T=\mathrm{Sp}_4(q)$, $q$ is even and $M\in\mathcal{C}_3$ is of type $\mathrm{Sp}_2(q^2):2$.  
According to~\cite[Table~8.14]{bray}, the $\mathrm{Aut}(T)$–conjugacy class of $M$ splits into two $T$–conjugacy classes, with representatives $M_1$ and $M_2$, where  
$M_1=\mathrm{SO}_4^-(q)$ lies in the Aschbacher class $\mathcal{C}_8$, and  
$M_2=\mathrm{Sp}_2(q^2):2$ lies in the Aschbacher class $\mathcal{C}_3$.  
These two $T$–classes are fused by a graph automorphism of $\mathrm{Sp}_4(q)$.
Let $V=\mathbb{F}_q^4$ be the natural module of $T$, equipped with the symplectic form defined by the matrix
\[
\begin{pmatrix}
0 & 1 & 0 & 0\\
1 & 0 & 0 & 0\\
0 & 0 & 0 & 1\\
0 & 0 & 1 & 0
\end{pmatrix}.
\]
Choose $x\in \mathrm{SL}_2(q)$ of order $q_2$ (where as usual $q_2=P[\Phi_{2e}(2)]$).  
Then
\[
P=\left\{
\begin{pmatrix}
a & 0\\
0 & b
\end{pmatrix}
\;\middle|\; a,b\in \langle x\rangle
\right\}
\]
is a subgroup of $T$ of order $q_2^2$. Let $Q$ be a Sylow $q_2$-subgroup of $T$ containing $P$. As $|T|=q^4(q^2-1)^2(q^2+1)$, we see that $P$ is the direct product of two cyclic groups of order $q_2^t$ (where $q_2^t$ is the largest power of $q_2$ dividing $q^2-1$) and $P$ is the subgroup $\Omega_1(Q)$ of $P$ consisting of the elements having order at most $q_2$.

Every element of $M_1=\mathrm{SO}_4^-(q)$ of order $q_2$ has eigenvalue $1$, and hence its Jordan form consists of two $1\times 1$ blocks and a single $2\times 2$ block.  
From this it follows that the $T$–conjugates of $M_1$ cover exactly $2(q_2-1)+1$ elements of $P$, that is,
\[
\left|\bigcup_{g\in T} \left(P\cap M_1^g\right)\right| = 2q_2 - 1.
\]
Let $\alpha$ be an automorphism of $T$ with $M_1^\alpha=M_2$. Now, $Q^\alpha\le T^\alpha=T$ and hence $Q^\alpha$ is a Sylow $q_2$-subgroup of $T$. By Sylow's theorem, there exists $x\in T$ with $Q^{\alpha t}=Q$. Thus $P=\Omega_1(Q)=\Omega_1(Q^{\alpha x})=(\Omega_1(Q))^{\alpha x}=P^{\alpha x}$ and $\alpha x$ normalizes $P$. Therefore,
\begin{align*}
\left(P\cap\bigcup_{g\in T}M_1^g\right)^{\alpha x}&=P^{\alpha x}\cap\bigcup_{g\in T}M_1^{g\alpha x}=P\cap\bigcup_{g\in T}M_1^{\alpha g^{\alpha}x}\\
&=P\cap \bigcup_{g\in T}M_2^{g^\alpha x}=P\cap \bigcup_{g\in T}M_2^g.
\end{align*}
This has shown that the $T$–conjugates of $M_2$ also cover $2q_2-1$ elements of $P$.  
Altogether, the $T$–conjugates of $M_1$ and $M_2$ (equivalently, the $\mathrm{Aut}(T)$–conjugates of $M$) cover
\[
2(2q_2-2)+1 = 4q_2 - 3
\]
elements of $P$.
If $4q_2 - 3 > |P| = q_2^2$, then there exists a cyclic subgroup $C\le P$ of order $q_2$ satisfying~\eqref{eqlem:0} and~\eqref{eqlem:1}.  
Moreover,~\eqref{eqlem:2} follows from Lemma~\ref{lem:cats}, since $T$ has Lie rank $3$.
If instead $4q_2 - 3 \ge q_2^2$, then $q_2 < 4$, so $q_2=3$.  
As $q_2=P[\Phi_{2e}(2)]$, this gives $2^e+1=3^\ell$ for some $\ell\in\mathbb{N}$.  
By Zsigmondy's theorem, this forces $e\in\{1,3\}$, and hence $q\in\{2,8\}$, both of which are excluded by our assumption that $|T|>|\mathbb{M}|$.

 Assume that either~\eqref{case4} or~\eqref{case5} occurs. Then by consulting the list of the maximal subgroups of $T$ (see for instance~\cite[Tables~8.16 and~8.43]{bray}) of order divisible by $q_4$ or $q_6$, we have $\gcd(|M|,q-1)=1$ when $T={}^2B_2(q)$ and $\gcd(|M|,(q-1)/2)=1$ when $T={}^2G_2(q)$. Let $q_1=\Phi_{e}(p)$ and let $C$ be a cyclic subgroup of $T$ of order $q_1$. Then $C$ satisfies~\eqref{eqlem:0} and~\eqref{eqlem:1} and moreover, it satisfies~\eqref{eqlem:2} by Lemma~\ref{lem:cats} provided we take $\theta(x)\ge \lambda(x,2)$, because $T$ has Lie rank $2$.

Assume~\eqref{case6} occurs. Then $q_3\le 13$ and hence, from Lemma~\ref{siegeltheorem}, we have $q\le \max\{\exp(\exp(13/c_{\Phi_3})),c'_{\Phi_3}\}$. 

Assume~\eqref{case7} occurs: we consider each line of Table~\ref{table:our} in turn.  Observe that here we need to bound both $q$ and $m$ from above with a function of $|C|$.
Suppose $(T,M)$ are as in Lines~1,~2, or~6 with $q$ even and $m>4$.  
From~\cite[Tables~3.5C,~3.5D,~3.5E, column~V]{kl}, the $\mathrm{Aut}(T)$-conjugacy class of $M$ coincides with the $T$-conjugacy class, which contradicts one of the assumptions we made earlier. Suppose $(T,M)$ are as in Lines~3,~4, or~5.  
From~\cite[Theorem~1.1]{Stewart}, there exists an absolute constant $c$ such that, whenever $2m>a$, we have 
\[
q_{2m} > 2m \exp\!\left(\frac{2m}{104 \log\log(2m)}\right).
\]
In particular, if $m$ is larger than some absolute constant $a'>a$, then $q_{2m} > 2m+3$.  
Thus, if we let $C$ be a cyclic subgroup of $T$ of order $q_{2m}$, then $C$ satisfies~\eqref{eqlem:0} and~\eqref{eqlem:1}; moreover, by Lemma~\ref{lem:cats}, it also satisfies~\eqref{eqlem:2} because the Lie rank of $T$ is bounded above by a function of $q_{2m}$.  

If $m \le a'$, then by Lemma~\ref{siegeltheorem} there exists a constant $a''$ such that, for all $q \ge a''$, we have $q_{2m} > 2a' + 3 \ge 2m + 3$, and the same argument applies.  
Finally, if $m \le a'$ and $q \le a''$, then $|T|$ is bounded above by an absolute constant, since the order of $T$ depends only on $q$ and $m$.

Suppose $(T,M)$ is as in Line~6 with $q$ odd, or with $q$ even and $m=4$.  
Let $C=\langle x\rangle$ be a cyclic subgroup of $T$ of order $q_m$, where the Jordan form of $x$ consists of two blocks of size $m\times m$.  
The existence of such an element $x$ follows from the standard embedding $
\mathrm{SO}_m^{-}(q)\times \mathrm{SO}_m^{-}(q) \le \mathrm{SO}_{2m}^{+}(q)$
together with the existence of Singer cycles in $\mathrm{SO}_m^{-}(q)$.  
Assume $m>4$.  
Then every $\mathrm{Aut}(T)$–conjugate of $M$ fixes a $1$–dimensional subspace of the natural $2m$–dimensional module for $T$.  
Consequently, no $\mathrm{Aut}(T)$–conjugate of $M$ meets $C$ in any non-trivial element.  
Thus $C$ satisfies conditions~\eqref{eqlem:0} and~\eqref{eqlem:1}.  
Moreover, by Lemma~\ref{lem:cats}, $C$ also satisfies~\eqref{eqlem:2}, since the Lie rank of $T$ is bounded above by a function of $q_m$. Assume $m=4$ that is, $T=\mathrm{P}\Omega_8^+(q)$ and $M\cong \Omega_7(q)$ (observe that when $q$ is even, we are identifying $\mathrm{Sp}_6(q)$ with $\Omega_7(q)$). The argument here is similar to the last part of case~\eqref{case3}.
According to~\cite[Table~8.20]{bray}, the $\mathrm{Aut}(T)$–conjugacy class of $M$ splits into six $T$–conjugacy classes when $q$ is odd and into three when $q$ is even, with representatives in the Aschbacher class $\mathcal{C}_1$ when $M$ is the stabilizer of a non-degenerate $1$-dimensional vector and in the Aschbacher class $\mathcal{S}$ when $M$ is embedded in $T$ via the spinor representations. Assume $q$ odd and let $M_1$ and $M_2$ be the $T$-representatives in $\mathcal{C}_1$ and let $M_3,\ldots,M_6$ be the representatives in $\mathcal{S}$.
Let $V=\mathbb{F}_q^8$ and let $P$ be an elementary abelian $q_4$-subgroup of $T$ of order $q_4^2$. The existence of $P$ follows from the embedding of $\mathrm{SO}_4^-(q)\times\mathrm{SO}_4^-(q)$ in $\mathrm{SO}_8^+(q)$, as above.
Every element of $M_1=\Omega_7(q)$ or $M_2$ of order $q_2$ has eigenvalue $1$, and hence its Jordan form consists of four $1\times 1$ blocks and a single $4\times 4$ block.  
From this it follows that the $T$–conjugates of $M_1$ cover exactly $2(q_4-1)+1$ elements of $P$, that is, $|\bigcup_{g\in T} \left(P\cap M_1^g\right)| = 2q_4 - 1$.
Altogether, the $\mathrm{Aut}(T)$–conjugates of $M$ cover at most 
$6(2q_2-2)+1 = 12q_4 - 11$
elements of $P$. If $12q_2 - 1 > |P| = q_4^2$, then there exists a cyclic subgroup $C\le P$ of order $q_4$ satisfying~\eqref{eqlem:0} and~\eqref{eqlem:1}.  
Moreover,~\eqref{eqlem:2} follows from Lemma~\ref{lem:cats}, since $T$ has Lie rank $4$.
If instead $12q_2 - 11 \ge q_4^2$, then $q_2 < 12$, so $q_4\in \{2,3,5,7\}$. As $4$ divides $q_4-1$, we get $q_4=5$. The argument when $q$ is even is entirely analogous.
\end{proof}

\section*{Acknowledgments}
 The second author is funded by the European Union via the Next
Generation EU (Mission 4 Component 1 CUP B53D23009410006, PRIN 2022, 2022PSTWLB, Group
Theory and Applications).
\end{document}